\documentclass[11pt]{amsart}

\allowdisplaybreaks

\usepackage{hyperref}

\usepackage{graphicx}

\usepackage{xcolor}

\usepackage{amsmath, amssymb, amsfonts, amsthm, fullpage}
\usepackage{graphics,graphicx}
\usepackage{accents}
\usepackage{color}

\usepackage{algorithm}
\usepackage{algorithmicx}
\usepackage{algpseudocode}

\usepackage[margin=0.9in]{geometry}

\newcommand{\norma}[1]{{\left\vert\kern-0.25ex\left\vert\kern-0.25ex\left\vert #1
    \right\vert\kern-0.25ex\right\vert\kern-0.25ex\right\vert}}
\newcommand{\sech}{\text{sech}}

\newcommand{\by}{\mathbf{y}}

\newcommand{\I}{\mathrm{i}}
\renewcommand{\Im}{\mathrm{Im}}
\renewcommand{\Re}{\mathrm{Re}}

\newcommand{\bb}{\mathbf{b}}
\newcommand{\bc}{\mathbf{c}}
\newcommand{\bo}{\mathbf{0}}

\newcommand{\bg}{\mathbf{g}}

\newcommand{\be}{\mathbf{e}}

\newcommand{\br}{\mathbf{r}}
\newcommand{\bs}{\mathbf{s}}
\newcommand{\bx}{\mathbf{x}}
\newcommand{\bF}{\mathbf{F}}

\newcommand{\bu}{\mathbf{u}}

\newcommand{\bz}{\mathbf{z}}
\newcommand{\bv}{\mathbf{v}}
\newcommand{\bA}{\mathbf{A}}
\newcommand{\bB}{\mathbf{B}}

\newcommand{\bD}{\mathbf{D}}
\newcommand{\bDF}{\mathbf{DF}}

\newcommand{\bG}{\mathbf{G}}

\newcommand{\bR}{\mathbf{R}}

\newcommand{\bI}{\mathbf{I}}
\newcommand{\bJ}{\mathbf{J}}
\newcommand{\bX}{\mathbf{X}}
\newcommand{\bY}{\mathbf{Y}}

\newcommand{\tbn}[1]{{\left\vert\kern-0.25ex\left\vert\kern-0.25ex\left\vert #1 \right\vert\kern-0.25ex\right\vert\kern-0.25ex\right\vert}}

\newtheorem{lemma}{Lemma}[section]

\newtheorem{theorem}{Theorem}[section]

\AtBeginDocument{%
  \setlength{\belowdisplayskip}{5pt} \setlength{\belowdisplayshortskip}{5pt}
  \setlength{\abovedisplayskip}{5pt} \setlength{\abovedisplayshortskip}{5pt}
}

\title{The complex-step Newton method and its convergence}

\author{Dimitrios Mitsotakis}
\address{\textbf{D.~Mitsotakis:} Victoria University of Wellington, School of Mathematics and Statistics, PO Box 600, Wellington 6140, New Zealand}
\email{dimitrios.mitsotakis@vuw.ac.nz}

\date{\today}

\begin{document}


\begin{abstract}
Considered herein is a modified Newton method for the numerical solution of nonlinear equations where the Jacobian is approximated using a complex-step derivative approximation. We show that this method converges for sufficiently small complex-step values, which need not be infinitesimal. Notably, when the individual derivatives in the Jacobian matrix are approximated using the complex-step method, the convergence is linear and becomes quadratic as the complex-step approaches zero. However, when the Jacobian matrix is approximated by the nonlinear complex-step derivative approximation, the convergence rate remains quadratic for any appropriately small complex-step value, not just in the limit as it approaches zero. This claim is supported by numerical experiments. Additionally, we demonstrate the method's robust applicability in solving nonlinear systems arising from differential equations, where it is implemented as a Jacobian-free Newton-Krylov method.
\end{abstract}

\maketitle

\section{Introduction}

The subject of this paper is the study of a modified Newton method known as complex-step Newton method for approximating roots of smooth functions. The complex-step Newton method when applied to systems of equations does not require the knowledge of the Jacobian matrix and therefore falls into the category of the Jacobian-free Newton methods. Jacobian-free Newton methods are significant for two primary reasons: (i) they operate effectively even when the Jacobian matrix (or the derivative in scalar cases) is unavailable, and (ii) they can be integrated with Krylov solvers to enhance their efficiency in solving systems of equations \cite{BH1986,BS1990,KK2004}. One example of such a method is the complex-step Newton method as it is demonstrated by \cite{KSPC22}. 
In what follows $h$ is assumed to be a fixed positive constant. Assume that we search for the solution $x^\ast$ of an equation $f(x)=0$ where $f:I\subset\mathbb{R}\to\mathbb{R}$ is appropriately smooth function and $I$ an interval that includes in its interior the root $x^\ast$. Then, the first derivative, $f'(x)$, can be approximated as $f'(x)\approx\Im f(x+\I h)/h$ \cite{LM1967,ST1998}. The parameter $h$ is referred to as the complex step, and the method itself is named after this parameter. Replacing the classical derivative by its complex-step approximation, the complex-step Newton iteration can be expressed in the following form:
$$x_{k+1}=x_k-h\frac{f(x_k)}{\Im f(x_k+\I h)} \quad \text{for $k=0,1,\dots$}\ ,$$
where $\I=\sqrt{-1}$ denotes the imaginary unit. 

Similarly, the complex-step Newton iteration can be applied to systems of equations of the form $\bF(\bx)=\bo$, where $\bx\in\mathbb{R}^n$ and $\bF:D\subset\mathbb{R}^n\to\mathbb{R}^n$ for $n\geq 2$. Specifically, the Jacobian matrix $\bDF(\bx)$ of $\bF$ can be approximated by the matrix with entries
$$[\bJ_h(\bx)]_{ij}=\frac{1}{h}\Im F_i(\bx+\I h\be_j)\quad \text{for any $h>0$}\ ,$$
where $\be_j$ denotes the usual basis of $\mathbb{R}^n$. The complex-step Newton method for the system $\bF(\bx)=\bo$ is the iteration
\begin{equation}\label{eq:intro1}
\bx^{(k+1)}=\bx^{(k)}-\bJ_h^{-1}(\bx^{(k)})\bF(\bx^{(k)})\quad \text{for k=0,1,2,\dots\ .}
\end{equation}
This formulation of the complex-step Newton method utilizes the approximate Jacobian matrix $\bJ_h$, and it is referred as the complex-step Jacobian Newton method.

Alternatively, as we shall see, we can approximate the matrix-vector product between the Jacobian matrix $\bDF(\bx)$ and a vector $\bv$ by the nonlinear function $\frac{1}{h}\Im \bF(\bx+\I h \bv)$ resulting to a Jacobian-free Newton method that can be expressed in the form of a fixed point method $\bx^{(k+1)}=\bx^{(k)}-\bu^{(k)}$ for $k=0,1,\dots$, where $\bu^{(k)}$ is the solution of the system
\begin{equation}
    \frac{1}{h}\Im \bF(\bx^{(k)}+\I h \bu^{(k)}) = \bF(\bx^{(k)})\ .
\end{equation}
The last equation viewed as an approximation of the equation $\bDF(\bx^{(k)})\bu^{(k)}=\bF(\bx^{(k)})$ can be solved efficiently using Krylov methods such as the Generalized Minimal Residual (GMRES) method without using the Jacobian matrix $\bDF(\bx^{(k)})$, \cite{BS1990}. Although these two variants of the complex-step Newton method appear very similar, they have different convergence properties.

The question of which values of $h$ lead to convergence, and at what rate, is critical for the complex-step Newton method. Addressing this issue can facilitate the method's use in lower-precision arithmetic \cite{Kelley2022} and increase confidence in its application.

Classical Newton's method is well-known for its quadratic rate of convergence, for which both local and semi-local convergence analyses have been established. In this work, we focus on local convergence analysis, where the convergence of the numerical method is studied under assumptions about the solution $\bx^\ast$ of the equation $\bF(\bx)=\bo$. Nevertheless, semi-local analysis, based on hypotheses related to the initial guess, can yield useful conclusions especially to error estimates and also provide insights of iterative methods when the processes is terminated in a finite number of iterations. For further details, refer to \cite{A2007,A2011,AMA2019,ortega1990}. 

In Section \ref{sec:convergence}, following a detailed derivation of the complex-step Newton iteration, we establish its convergence for sufficiently small values of $h$. Specifically, for the Jacobian Newton method, we show that the convergence rate becomes quadratic as $h$ approaches zero. On the other hand, we demonstrate that the Jacobian-free variant of the complex-step Newton method achieves a quadratic rate of convergence even for moderately small values of $h$. We empirically validate these findings and conclude with a study of the influence of the parameter $h$ across various applications.

In Section \ref{sec:application} we demonstrate the use of the complex-step Jacobian-free Newton method for the solution of nonlinear systems obtained during the numerical integration of stiff ordinary differential equations by a symplectic Runge-Kutta method, namely, the fourth-order Gauss-Legendre Runge-Kutta method of order four \cite{SC1994}. Moreover, in the same section we consider the discrete nonlinear Schr\"{o}dinger (DNLS) equation \cite{Kev}. First we obtain a steady-state solution of this equation by solving the corresponding nonlinear system of equations using the complex-step Newton method. Then we employ the implicit fourth-order Gauss-Legendre Runge-Kutta method for the time discretization of the ODE system of the corresponding DNLS equation. Due to its symplectic nature, the particular time-integration method demonstrates excellent conservation properties by conserving the $\ell_2$-norm of the solution (quadratic functional) while almost conserving the Hamiltonian \cite{HLW2006}. 

\section{Convergence}\label{sec:convergence}

\subsection{Scalar equations}\label{sec:scalar}

First we review the derivation of the complex-step Newton iteration. Consider the equation $f(x)=0$ where $f:I\to\mathbb{R}$ is a real analytic function in a closed interval $I\subset \mathbb{R}$ appropriately chosen such that $f'(x)\not=0$ for all $x\in I$. We assume that the equation $f(x)=0$ has a unique solution $x^\ast\in I$. Given $x_0\in I$, we define the classical Newton iteration as
\begin{equation}\label{eq:csnewton1}
    x_{k+1}=x_k-\frac{f(x_k)}{f'(x_{k})}\quad \text{for $k=0,1,\dots$}\ .
\end{equation}
In order to estimate alternatives of the first derivative using complex number arithmetic (\cite{LM1967,ST1998}) we consider the Taylor expansion of the function $g(h)=\Im f(x+\I h)$ for $x\in I$ and $h\in \mathbb{R}$ around $h=0$
\begin{equation}\label{eq:onedtaylor}
g(h)=g(0)+g'(0)h+\frac{g''(0)}{2}h^2+\frac{g'''(c_h)}{6}h^3\ ,
\end{equation} where $0<c_h<h$. 
Obviously it is $g(0)=0$. Writing $f(x+\I h)=r(x,h)+\I s(x,h)=r(x,h)+\I g(h)$ we have that $g'(h)=s_h(x,h)$, where the subscript $h$ here denotes partial differentiation with respect of $h$. On the other hand, $\frac{d}{dh}f(x+\I h)=\I f'(x+\I h)=r_h(x,h)+\I s_h(x,h)=r_h(x,h)+\I g'(h)$. Thus, for $h=0$ we have $\I f'(x)=r_h(x,0)+\I g'(0)$. Taking the imaginary part of both sides we obtain $g'(0)=f'(x)$. Note that the derivative in the last formula coincides with the real derivative $f'(x)$ due to the path-independence of the complex derivatives. Similarly, since $\frac{d^2}{dh^2}f(x+\I h)=-f''(x+\I h)=r_{hh}(x,h)+\I s_{hh}(x,h)$ we have $g''(0)=s_{hh}(x,0)=0$. Finally, $\frac{d^3}{dh^3}f(x+\I h)=-\I f'''(x+\I h)=r_{hhh}(x,h)+\I g'''(h)$ and thus $g'''(c_h)=-\Im[\I f'''(x+\I c_h)]$. Combining the previous formulas in (\ref{eq:onedtaylor}) we obtain to the relation
\begin{equation}\label{eq:cstepd}
f'(x)=\frac{1}{h}\Im f(x+\I h)+R^x(c_h)h^2\ ,
\end{equation}
with $R^x(c_h)=\tfrac{1}{6}\Im[\I f'''(x+\I c_h)]$.
The approximation of the first derivative $f'(x)$ by $\Im f(x+\I h)/h$ is called the complex-step approximation and it does not suffer from cancellation errors like the standard finite difference formulas for values of $h$ even of the order $10^{-20}$ and less, \cite{LM1967,ST1998}. For this reason we can approximate the first derivative up to the machine precision. Based on this approximation, the complex-step Newton iteration is defined as
\begin{equation}\label{eq:cnewtonsc}
x_{k+1}=x_k-h\frac{f(x_k)}{\Im f(x_k+\I h)} \quad \text{for $k=0,1,\dots$}\ .
\end{equation}

The assumption that $f$ is real analytic in an interval containing the root $x^\ast$ can be relaxed; however, it ensures the extension of $f$ into a complex analytic function within a region surrounding the root $x^\ast$. For example if $f(x)=\sum_n a_n (x-x^\ast)^n$ for $x\in D$, where $D$ is closed interval that includes $x^\ast$ in its interior, then there is an $\varepsilon>0$ such that the analytic complex extension of $f$ can be defined as $f(z)=\sum_n a_n (z-x^\ast)^n$ in $D\times [-\I\varepsilon,\I\varepsilon]$. This extension guarantees that the derivatives of $f$ remain uniformly bounded within a closed region of the complex domain, as per Cauchy's integral formula \cite[Section 2.6.2]{AF2003}. The assumption of uniformly bounded first and third derivatives is necessary, as elucidated in the subsequent proofs. In general, we will denote by $C$ any generic positive constant independent of $h$.

\begin{lemma}\label{lem:derivbound}
If the function $f$ is real analytic in a closed interval $I$ and $x^\ast$ in the interior of $I$ with $f(x^\ast)=0$ and $f'(x^\ast)\not=0$, then there are $\bar{h}>0$ and $m,M>0$  independent of $h$ such that $m\leq \frac{1}{h}|\Im f(x+ih)|\leq M$ for all $0<h<\bar{h}$.
\end{lemma}
\begin{proof}
Since $f'(x^\ast)\not=0$ there is a neighborhood $J\subset I$ with $x^\ast\in J$ and $m',M'>0$ such that $m'< |f'(x)|< M'$ for all $x\in I'$. Furthermore, due to the analyticity of $f$, there is a closed neighborhood $\mathbb{D}\subset \mathbb{C}$ around $x^\ast$ and a $K>0$ such that $|f'''(z)|<K$ for all $z\in \mathbb{D}$.

From (\ref{eq:cstepd}) we have 
$$
\left|\tfrac{1}{h}\Im f(x+\I h)\right|\leq |f'(x)|+|R^x(c_h)|h^2<M'+|R^x(c_h)|h^2\ ,
$$
and thus there is an $h_1>0$ such that $|\Im f(x+\I h)|/h\leq 2M'$ for all $0<h<h_1$. 
Similarly, 
$$
\left|\tfrac{1}{h}\Im f(x+\I h)\right|\geq |f'(x)|-|R^x(c_h)|h^2>m'-|R^x(c_h)|h^2\ ,
$$
which yields that there is an $h_2>0$ such that $|\Im f(x+\I h)|/h\geq m'/2$, for $0<h<h_2$. We take $\bar{h}=\min(h_1,h_2)$, $m=m'/2$ and $M=2M'$. 
\end{proof}

\subsection{Convergence for scalar equations}

Now we can prove that the complex-step Newton iteration converges. Moreover, we show that the convergence is linear for small values of $h$, and becomes quadratic when $h\to 0$.

\begin{theorem}\label{thm:scalar}
Given a real analytic function $f$ in a closed interval $I$ and $x^\ast$ in the interior of $I$ with $f(x^\ast)=0$ and $f'(x^\ast)\not=0$, then there is $h_0>0$ and an interval $D\subset I$ such that the complex-step Newton iteration $x_k$ of (\ref{eq:cnewtonsc}) converges to $x^\ast$ for all $x_0\in D$ and for all $0<h<h_0$. Furthermore, 
\begin{equation}\label{eq:cratesc}
\lim_{k\to\infty}\left[\lim_{h\to 0} \frac{x_{k+1}-x^\ast}{(x_k-x^\ast)^2}\right]=\frac{f''(x^\ast)}{2f'(x^\ast)}\ .
\end{equation}
\end{theorem}

\begin{proof}
First we establish the convergence of the method by recalling that the conditions of Lemma \ref{lem:derivbound} imply the existence of constants $m, M>0$, such that $m<|f'(x)|<M$ for all $x\in J$ where $J\subset I$. Additionally, there exists a closed neighborhood $\mathbb{D}\subset \mathbb{C}$ of $x^\ast$ and a constant $K>0$ such that $|f'''(z)|<K$ for all $z\in \mathbb{D}$. Having established in Lemma \ref{lem:derivbound} that the complex-step approximation of the derivative is bounded by a bound that is independent of $h$ for $0<h<\bar{h}$, we rewrite (\ref{eq:cstepd}) as
\begin{equation}\label{eq:fracbound}
h\frac{f'(x)}{\Im f(x+\I h)}=1+\tilde{R}(c_h)h^2\ ,
\end{equation}
where $\tilde{R}(c_h)=\frac{R^x(c_h)}{\tfrac{1}{h} \Im f(x+\I h)}$. 
Using the lower bound of the complex-step approximation, we have
$|\tilde{R}(c_h)|\leq |R^x(c_h)|/m$.
Thus we have that $|\tilde{R}(c_h)|\leq K/6m$ where the bound is independent of $h$.

The complex-step Newton iteration (\ref{eq:cnewtonsc}) is written as a fixed point equation $x_{k+1}=G_h(x_k)$, for $k=0,1,\dots$ where
$$G_h(x)=x-h\frac{f(x)}{\Im f(x+\I h)}\ .$$
Taking into account that $f$ is analytic, (\ref{eq:fracbound}) and that $f(x^\ast)=0$, we have that
\begin{equation}\label{eq:deriv2}
G'_h(x^\ast)=1-h\frac{f'(x^\ast)}{\Im f(x^\ast+\I h)}=-\tilde{R}(c_h)h^2\ .
\end{equation}
From Lemma \ref{lem:derivbound} and the previous comments we deduce that there is an $\bar{h}>0$ such that for $h\in(0,\bar{h})$ the derivative $|G'_h(x^\ast)|\leq Ch^2<1$ and thus $G_h$ is a contraction, and we conclude by the fixed point theorem (cf. \cite{ortega1990}) that there is a neighborhood $D$ of $x^\ast$ such that for any $x_0\in D$ the iteration $x_{k+1}=G_h(x_k)$ converges to $x^\ast$ as $k\to\infty$. In the sequel we denote the error of the approximation by $e_k=x_k-x^\ast$, and thus we have that $\lim_{k\to\infty}e_k=0$. 

Since $f$ is analytic, we have the following Taylor expansions for the function $f$,
\begin{equation}
f(x_k) = f(x^\ast)+(x_k-x^\ast)f'(x^\ast)+\frac{(x_k-x^\ast)^2}{2}f''(c_k)=e_k f'(x^\ast)+\frac{e_k^2}{2}f''(c_k)\ , 
\end{equation}
and its derivative
\begin{equation}
\frac{1}{h}\Im f(x_k+\I h) = f'(x_k)-R^x(c_h)h^2=f'(x^\ast)+e_kf''(d_k)-R^x(c_h)h^2\ ,
\end{equation}
where $c_k$ and $d_k$ are constants between $x^\ast$ and $x_k$, while $0<c_h<h$ as before. Substituting these expansions into (\ref{eq:cnewtonsc}) we obtain
\begin{equation}
x_{k+1}=x_k-\frac{e_k f'(x^\ast)+\frac{e_k^2}{2}f''(c_k)}{f'(x^\ast)+e_k f''(d_k)-R^x(c_h)h^2}\ .
\end{equation}
Subtracting $x^\ast$ from both sides and denoting $e_k=x^k-x^\ast$ the error at the $k$-th iteration, the last equation yields
\begin{align}
e_{k+1} &= e_k - \frac{e_k f'(x^\ast)+\frac{e_k^2}{2}f''(c_k)}{f'(x^\ast)+e_k f''(d_k)-R^x(c_h)h^2}\nonumber \\
&=e_k^2\frac{f''(d_k)-\tfrac{1}{2}f''(c_k)}{f'(x^\ast)+e_k f''(d_k)-R^x(c_h)h^2}-e_k \frac{R^x(c_h)h^2}{f'(x^\ast)+e_k f''(d_k)-R^x(c_h)h^2}\ . \label{eq:ntr2}
\end{align}
From this we extract two pieces of information. First, dividing by $e_k$ and taking the limit $k\to \infty$ we have that
$$\lim_{k\to\infty}\frac{e_{k+1}}{e_{k}}=-\frac{R^x(c_h)h^2}{f'(x^\ast)-R^x(c_h)h^2}\ .$$
Let $B=\sup_{z\in\mathbb{D}}|f'''(z)|>0$ where $\mathbb{D}$ is a neighborhood of $x^\ast$, we define $h_1=\sqrt{3|f'(x^\ast)|/B}$. Then for $0<h<h_0=\min\{\bar{h},h_1\}$ we have that
$$\frac{|R^x(c_h)h^2|}{|f'(x^\ast)-R^x(c_h)h^2|}\leq \frac{1}{\frac{|f'(x^\ast)|}{|R^x(c_h)h^2|}-1}<1\ .$$ Therefore, the complex-step Newton method converges for all $h<h_0$.

Moreover, we see that the error behaves better as $h\to 0$. From equation (\ref{eq:deriv2}) we have that $\lim_{h\to 0}G'_h(x^\ast)=0$, and thus 
$$\lim_{h\to 0}\frac{x_{k+1}-x^\ast}{x_k-x^\ast}=\lim_{h\to 0} \frac{G_h(x_k)-G_h(x^\ast)}{x_k-x^\ast}=\lim_{h\to 0} \left[\frac{G_h(x_k)-G_h(x^\ast)}{x_k-x^\ast}-G'_h(x^\ast)\right]\ . $$
Thus, 
$$ \lim_{k\to\infty}\left[\lim_{h\to 0}\frac{e_{k+1}}{e_k}\right]=0\ .$$
Additionally, we have from (\ref{eq:ntr2}) that
$$\frac{e_{k+1}}{e_{k}^2}=\frac{f''(d_k)-\tfrac{1}{2}f''(c_k)}{f'(x^\ast)+e_k f''(d_k)-R^x(c_h)h^2}- \frac{R^x(c_h)h^2}{e_k[f'(x^\ast)+e_k f''(d_k)-R^x(c_h)h^2]}\ , $$
and more precisely, by taking $h\to 0$ we obtain
$$\lim_{h\to 0}\frac{e_{k+1}}{e_{k}^2}=\frac{f''(d_k)-\tfrac{1}{2}f''(c_k)}{f'(x^\ast)+e_k f''(d_k)}\ ,$$
and thus
$$\lim_{k\to\infty}\left[\lim_{h\to 0} \frac{e_{k+1}}{e_{k}^2}\right]=\frac{f''(x^\ast)}{2f'(x^\ast)}\ ,$$
since both $c_k\to x^\ast$ and $d_k\to x^\ast$ as $k\to \infty$, which completes the proof.
\end{proof}

\subsection{Experimental convergence rate}

This result implies that for sufficiently small values of $h \ll 1$, the observed convergence will be quadratic. We experimentally study the convergence and the influence of the complex-step $h$ on the convergence using the simple equation $f(x) = 0$ with $f(x) = x(e^{x/2} + 1)$, where $x^\ast = 0$ is an obvious root. Moreover, $f'(x^\ast) \neq 0$, and the above theory is applied to this case in a straightforward manner. 
\begin{figure}[ht!]
  \centering
\includegraphics[width=\columnwidth]{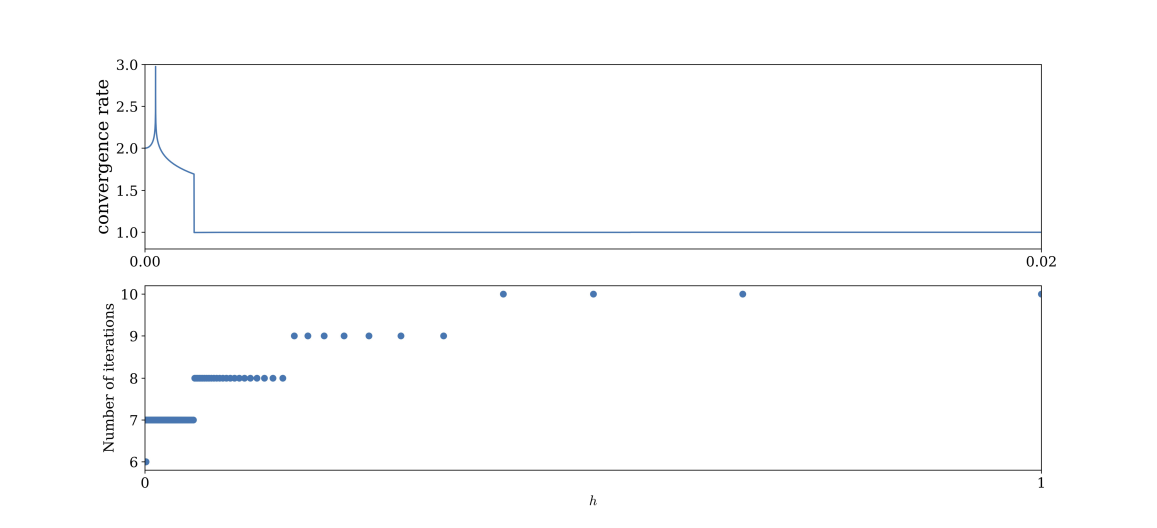}
  \caption{Convergence rate as a function of $h$ for a scalar function $f$}
  \label{fig:rates1}
\end{figure}

To study the influence of $h$, we consider the values $h=2/n$ for $n=1,2,\dots,10^6$. Figure \ref{fig:rates1} presents the experimental convergence rate $\log|e_{k+1}/e_{k}|/\log|e_{k}/e_{k-1}|$, where $e_k=x_k-x^\ast$ is the error of the $k$-th iteration as a function of $h$ for the last three iterations before the error $|x_k-x^\ast|$ becomes less than the prescribed tolerance. As a convergence criterion we used the inequality $|x_k-x^\ast|\leq 10^{-14}$ and initial guess $x_0=2.5$. 

In Figure \ref{fig:rates1}, we observe that the method converged clearly with rate $1$ for large values of $h$  while it accelerated after a threshold value. The maximum convergence rate was observed for $h\approx 0.00023635$.  As $h$ became much smaller than $1$, the convergence rate was stabilized to $2$ as predicted by the Theorem \ref{thm:scalar}, while for values of $h$ of $O(1)$ we observe that the convergence rate is 1.  In all cases the error became smaller than the tolerance within 11 iterations. Similar behaviour was observed when we chose the tolerance to be $10^{-15}$. We proceed next with the study of the complex-step Newton method for systems of algebraic equations, where the situation is slightly different.

\subsection{Systems of equations}\label{sec:systems}

Consider a system of equations of the form $\bF(\bx)=\bo$ where $\bF: U\subset \mathbb{R}^n\to \mathbb{R}^n$ is a vector valued function $\bF=(F_1,F_2,\dots,F_n)^T$ with $n\geq 2$. We assume that $F_i$ are real analytic functions on $U$ and $\bx=(x_1,x_2,\dots, x_n)^T$. In this way, all the derivatives of $F_i$ can be extended to complex analytic functions in a closed neighborhood of $\bx^\ast$ similarly to the one-dimensional case. Moreover, $\bF$ is Fr\'{e}chet-differentiable in a neighborhood of $\bx^\ast$.  Given $\bx^{(0)}\in U$, we define the classical Newton iteration as
$$\bx^{(k+1)}=\bx^{(k)}-\bDF^{-1}(\bx^{(k)}) \bF(\bx^{(k)})\quad \text{for $k=0,1,\dots$}\ ,$$
where the $\bDF$ is the Jacobian matrix with entries the derivatives $[\bDF]_{i,j}=\partial_j F_i$ where $\partial_j =\frac{\partial}{\partial x_j}$.
For convenience in the implementation we usually express the general Newton iteration as a linear system of equations in the form:
\begin{equation}\label{eq:classnewt}
\bDF(\bx^{(k)})\bu^{(k)} = \bF(\bx^{(k)}),\qquad \bx^{(k+1)}=\bx^{(k)}-\bu^{(k)},\quad \text{for $k=0,1,\dots$}\ .    
\end{equation}
Working similarly as in Section \ref{sec:scalar}, we assume that the function $\bF$ can be extended in the neighborhood of $\bx^\ast$ in the complex space and that for $\bz=\bx+\I\by$ we can write $\bF(\bz)=\br(\bx,\by)+\I \bs(\bx,\by)$, where $\bs(\bx,\by)=\Im\bF(\bz)$. We set $\bg(h)=\Im \bF(\bx+\I h \bv)=\bs(\bx,h\bv)$ where $\bv\in \mathbb{R}^n$. This implies that $\bg'(h)=\nabla_\by \bs(\bx,h\bv)\bv$, where $\nabla_\by\bs$ denotes the Jacobian matrix of $\bs(\bx,\by)$ with respect to vector $\by$. On the other hand, if $\nabla_\bz \bF(\bz)$ denotes the complex Jacobian matrix with entries $[\partial F_i/\partial z_j]_{ij}$, then $\partial_h \bF(\bx+\I h \bv)= \I \nabla_\bz \bF(\bx+\I h \bv)\bv=\nabla_\by \br(\bx,h\bv)\bv+\I\nabla_\by \bs(\bx,h\bv)\bv$. Thus, we have that $\bg'(0)=\nabla_\by \bs(\bx,0)\bv=\nabla_\bz \bF(\bx)\bv=\bD\bF(\bx)\bv$, where the last equality holds because of the path-independence of complex derivatives. 

Then, the Taylor expansion of $\bg$ about $h=0$ will be
$$\bg(h)=\bg(0)+\bg'(0)h+\frac{\bg''(0)}{2}h^2+\frac{\bg'''(c_h)}{6}h^3\ ,$$ here $0<c_h<h$, while $\bg(0)=\bo$, $\bg'(0)=\bDF(\bx)\bv$, $\bg''(0)=\bo$ and we denote the third derivative vector $-\bg'''(c_h)/6$ as $\bR^\bx(c_h)=\bR^\bx(c_h;\bv)$ that depends on $\bv$ and has $\ell$-entry (cf. \cite{Taylor2012})
$$[\bR^\bx(c_h;\bv)]_\ell=-\frac{[\bg'''(c_h)]_\ell}{6}=\frac{1}{6}\sum_{i,j,k=1}^n\Im\left[\I\frac{ \partial^3 F_\ell(\bx+\I c_h \bv)}{\partial v_i\partial v_j \partial v_k}v_iv_jv_k\right]\ .$$
Taking absolute values we obtain the bounds
\begin{equation}\label{eq:residual1}
|[\bR^\bx(c_h;\bv)]_\ell|\leq \frac{1}{6}\sum_{i,j,k=1}^n\left|\Im\left[\I \frac{ \partial^3  F_\ell(\bx+\I c_h \bv)}{\partial v_i\partial v_j \partial v_k}\right]\right|\|\bv\|^3\ .
\end{equation}

Furthermore, we assume that the third-order derivatives of $\bF(\bz)$ are uniformly bounded by a bound independent of $h$, which can be ensured by the analyticity of $\bF$ in a closed region of the complex spaces that includes in its interior the solution $\bx^\ast$. If for example we assume that there is a constant $K\geq 0$ (independent of $h$) such that $|\partial^{3}F_\ell(\bz)/\partial z_i\partial z_j\partial z_k|\leq K$ uniformly for all $\bz$ in a closed neighborhood of $\bx^\ast$ and for all $i,j,k,\ell=1,2,\dots, n$, then we have that
\begin{equation}\label{eq:normres}
\|\bR^\bx(c_h;\bv)\|\leq C\|\bv\|^3\ ,
\end{equation}
where $C>0$ is independent of $h$.
This leads to the approximation 
\begin{equation}\label{eq:cstepvec}
\bDF(\bx)\bv=\frac{1}{h}\Im \bF(\bx+\I h\bv)+ \bR^\bx(c_h;\bv)h^2\ .
\end{equation}
This leads to the approximation of the Jacobian matrix $\bD\bF(\bx)$  by the nonlinear operator $\bD_h\bF(\bx):\mathbb{R}^n\to \mathbb{R}^n$, defined by:
\begin{equation}\label{eq:approx1}
\bD_h\bF(\bx)\bv\doteq \frac{1}{h} \Im \bF(\bx+\I h\bv)\ .
\end{equation}
This approximation leads to the complex-step Jacobian-free Newton method.
An alternative approximation that leads to the complex-step Jacobian Newton method is the following: Denote the unit normal basis vectors of $\mathbb{R}^n$ by $\be_j$, and write $\bv=\sum_{j=1}^n v_j \be_j$. Then formula (\ref{eq:cstepvec}) implies that
\begin{equation}\label{eq:JNM1a}
\bDF(\bx)\bv=\sum_{j=1}^n v_j \bDF(\bx) \be_j =\sum_{j=1}^n v_j \frac{1}{h}\Im \bF(\bx+\I h\be_j)+ \sum_{j=1}^n v_j \bR^\bx(c_h;\be_j) h^2\ .
\end{equation}
It is natural then to consider the approximation of the Jacobian matrix $\bDF(\bx)$ with entries:
\begin{equation}
\partial_j F_i(\bx)\approx\frac{1}{h}[\Im \bF(\bx+\I h\be_j)]_i\ ,
\end{equation}
and define the complex-step approximation $\bJ_h(\bx)$ of the Jacobian matrix  with entries
$$[\bJ_h(\bx)]_{ij}=\frac{1}{h}[\Im \bF(\bx+\I h\be_j)]_i=\frac{1}{h}\Im F_i(\bx+\I h\be_j)\quad \text{for any $h>0$}\ .$$
Moreover, defining the matrix $\bB^\bx$ with columns $\bR^x(c_h,\be_j)$, we write (\ref{eq:JNM1a}) in the matrix form
\begin{equation}\label{eq:JNapp}
\bDF(\bx)\bv=\bJ_h(\bx)\bv+ h^2\bB^\bx\bv\ .
\end{equation}

Based on the last approximation of the Jacobian matrix, we define the complex-step Jacobian Newton method for the system $\bF(\bx)=\bo$ to be the iteration
\begin{equation}\label{eq:twodcsnewton}
\bJ_h(\bx^{(k)})\bu^{(k)}=\bF(\bx^{(k)}),\qquad \bx^{(k+1)}=\bx^{(k)}-\bu^{(k)}, \quad \text{for $k=0,1,2,\dots$}\ .
\end{equation}
where $\bx^{(0)}$ is a given initial guess of the exact solution $\bx^{\ast}$. 

On the other hand, based on the relation (\ref{eq:approx1}), we define the complex-step Jacobian-free Newton method to be the iteration such that for given initial guess $\bx^{(0)}$ 
\begin{equation}\label{eq:linearsysc}
\bD_h\bF(\bx^{(k)})\bu^{(k)}=\bF(\bx^{(k)}),\qquad \bx^{(k+1)}=\bx^{(k)}-\bu^{(k)}, \quad \text{for $k=0,1,2,\dots$}\ .
\end{equation}

We emphasise that the system (\ref{eq:twodcsnewton}) is a linear system of equations, while (\ref{eq:linearsysc}) is a nonlinear system. 

\subsection{Convergence for systems of equations}

We next verify that both iterations are well-defined and converge to the root $\bx^\ast$ when the initial guess lies within an appropriate neighborhood of $\bx^\ast$. Moreover, we estimate its order of convergence. In contrast to the one-dimensional case, the convergence rate of the Jacobian-free Newton method is quadratic for any value of $h$ at which the method can converge while the complex-step Jacobian Newton method share the same properties with the scalar version. To prove this, we follow \cite{ortega1990}, and in particular we start by generalizing the Lemma 8.1.5 of \cite{ortega1990}. In general, we will denote by $C$ any generic positive constant independent of $h$. 

Specifically, we have the following estimates for the matrix $\bJ_h$:
\begin{lemma}\label{lem:ostr1a}
If $\bF:U\subset \mathbb{R}^n\to \mathbb{R}^n$ has real-analytic entries in a neighborhood $U\subset \mathbb{R}^n$, then there is a neighborhood $D\subset \mathbb{R}^n$ and $h_0>0$ such that for any $\bx$, $\by\in D$ and $0<h<h_0$ 
\begin{equation}\label{eq:ostr2a}
\|\bF(\by)-\bF(\bx)-\bJ_h(\bx)(\by-\bx)\|\leq C\left(\|\by-\bx\|^2+h^2\|\by-\bx\|\right)\ ,
\end{equation}
for some positive constant $C$ independent of $h$.
Moreover, there is a constant $C'$ such that 
\begin{equation}\label{eq:ostr3a}
\|(\bJ_h(\by)-\bJ_h(\bx))(\by-\bx)\|\leq C'\left(\|\by-\bx\|^2+h^2\|\by-\bx\|\right)\ .
\end{equation}
\end{lemma}
\begin{proof}
We note that by the analyticity of $\bF$ we deduce that there is a constant $C>0$ such that 
\begin{equation}\label{eq:olip}
\|\bDF(\by)-\bDF(\bx)\|\leq C\|\by-\bx\|\ ,
\end{equation}
for all $\bx$, $\by$ in a neighborhoud of $\bx^\ast$ \cite[Corolary 3.3.5]{ortega2000}. Moreover, there is a  neighborhood $\mathbb{D}\subset \mathbb{C}^n$ of $\bx^\ast$ and a constant $K\geq 0$ independent of $h$ such that $|\partial^{3}F_\ell(\bz)/\partial z_i\partial z_j\partial z_k|\leq K$ uniformly for all $i,j,k,\ell=1,2,\dots, n$ and $\bz\in \mathbb{D}$.

With the help of (\ref{eq:JNapp}), we write the quantity in (\ref{eq:ostr2a}) as
$$\bF(\by)-\bF(\bx)-\bJ_h(\bx)(\by-\bx)=\bF(\by)-\bF(\bx)-\bDF(\bx)(\by-\bx)-h^2\bB^\bx(\by-\bx)\ .$$
From the definition of $\bB^{\bx}$, we have
$$\|\bB^\bx(\by-\bx)\|\leq C \|\by-\bx\|\ .$$
Note that (cf. \cite[Lemma 8.1.2]{ortega1990})
$$\bF(\by)-\bF(\bx)-\bDF(\bx)(\by-\bx)=\int_0^1[\bDF(\bx+t(\by-\bx))-\bDF(\bx)](\by-\bx)~dt\ .$$
Therefore, we have
$$\begin{aligned}
\|\bF(\by)-\bF(\bx)-\bJ_h(\bx)(\by-\bx)\|&\leq \| \bF(\by)-\bF(\bx)-\bDF(\bx)(\by-\bx)\|+ C h^2\|\by-\bx\|\\
&\leq \int_0^1\|\bDF(\bx+t(\by-\bx))-\bDF(\bx)\|\|\by-\bx\|dt+ C h^2\|\by-\bx\|\\
&\leq C\left(\|\by-\bx\|^2+h^2\|\by-\bx\|\right)\ .
\end{aligned}$$
Similarly, using (\ref{eq:JNapp}) and (\ref{eq:olip}), we have
$$
\|(\bJ_h(\bx)-\bJ_h(\by))(\by-\bx)\| \leq \|(\bDF(\bx)-\bDF(\by))(\by-\bx)\|+h^2\|\bB^\bx(\by-\bx)-\bB^\by(\by-\bx)\|\ ,
$$
from which we obtain the inequality (\ref{eq:ostr3a}).
\end{proof}

We are now ready to prove the convergence of the complex-step Jacobian Newton iteration.

\begin{theorem}\label{thm:vecconverg}
Let the function $\bF=(F_1,F_2,\dots,F_n):U\subset \mathbb{R}^n\to \mathbb{R}^n$ with $F_i$ real analytic functions in a closed domain $U$ and $\bx^\ast$ is the unique vector in the interior of $U$ such that $\bF(\bx^\ast)=\bo$. If the matrix $\bDF(\bx^\ast)$ is invertible, then there is a neighborhood $D\subset \mathbb{R}^n$ of $\bx^\ast$ and $h_0>0$ such that for any initial guess $\bx^{(0)}\in D$ and for all $0<h<h_0$, the complex-step Jacobian Newton iteration (\ref{eq:twodcsnewton}) is well-defined and converges to $x^\ast$. Moreover, there is a constant $C\geq0$ independent of $h$ such that
\begin{equation}\label{eq:errest}
\lim_{k\to \infty}\left[\lim_{h\to 0}\frac{\|\bx^{(k+1)}-\bx^\ast\|}{\|\bx^{(k)}-\bx^\ast\|^2}\right]=C\ .
\end{equation}
\end{theorem}

\begin{proof}

First we note that the complex-step Jacobian Newton method (\ref{eq:twodcsnewton}) can be expressed as a fixed point method of the form $\bx^{(k+1)}=\bG_h(\bx^{(k)})$ with
\begin{equation}\label{eq:cstepnewtvec}
\bG_h(\bx)=\bx-\bJ_h^{-1}(\bx)\bF(\bx)\quad \text{for all $\bx\in U$}\ .
\end{equation}

The assumption that $F_i$ are real analytic for all $i=1,2,\dots,n$, implies that $\bF$ is Fr\'{e}chet differentiable in a neighborhood $U$ of $\bx^\ast$ and that the derivatives of $F_i$ can be extended to complex functions in a closed neighborhood $\mathbb{D}\subset \mathbb{C}^n$ that contains $\bx^\ast$ in its interior. Thus, we can assume that there is a constant $K>0$ independent of $h$ such that $|\partial^{3}F_\ell(\bz)/\partial z_i\partial z_j\partial z_k|\leq K$ for all $i,j,k,\ell=1,2,\dots, n$ and $\bz\in \mathbb{D}$.
Moreover, we have that complex-step Jacobian matrix $\bJ_h(\bx)$ is continuous.  

Since 
\begin{equation}\label{eq:limit}
\bDF(\bx^\ast)=\lim_{h\to 0} \bJ_h(\bx^\ast)\ ,
\end{equation} we have that there is an $h_1>0$ such that for $0<h<h_1$ the $\det[\bJ_h(\bx^\ast)]\not=0$ and thus is invertible. Moreover from \cite[Section 8.1.8]{ortega1990} we deduce that $\bJ_h(\bx)$ is invertible in a closed ball $S(\bx^\ast,\delta)\subset U$ with center $\bx^\ast$ and appropriate radius $\delta>0$. Thus, the iteration (\ref{eq:cstepnewtvec}) is well-defined. Also, we have that the Fr\'{e}chet derivative of $\bG_h$ is the matrix
\begin{equation}\label{eq:derivative}
\bG'_h(\bx^\ast)=\bI-\bJ_h^{-1}(\bx^\ast)\bDF(\bx^\ast)\ .
\end{equation}
If $\lambda\in \rho(\bG_h'(\bx^\ast))$ is an eigenvalue of the derivative matrix $\bG_h'(\bx^\ast)$ with corresponding eigenvector $\bv$, then we have $\bG_h'(\bx^\ast)\bv-\lambda \bv=\bo$. This yields
\begin{equation}\label{eq:alternform}
(1-\lambda)\bJ_h(\bx^\ast)\bv=\bDF(\bx^\ast)\bv\ .
\end{equation} 
We substitute (\ref{eq:cstepvec}) into (\ref{eq:alternform}) to obtain
$(1-\lambda)\bDF(\bx^\ast)\bv-(1-\lambda) h^2\bB^\bx\bv =\bDF(\bx^\ast)\bv$ or equivalently $$\lambda[h^2 \bB^\bx \bv -\bDF(\bx^\ast)\bv]=h^2 \bB^\bx\bv\ .$$
Taking the Euclidean norm on both sides we obtain
\begin{equation}\label{eq:eigenvl}
|\lambda|= h^2\frac{\|\bB^\bx\bv\|}{\|h^2 \bB^\bx\bv -\bDF(\bx^\ast)\bv\|}\ .
\end{equation}
Therefore there is an $h_2>0$ such that for $0<h<h_2$ the eigenvalue $|\lambda|\leq\sigma<1$. From Ostrowski's Theorem \cite[Theorem 8.1.7]{ortega1990} we have that $\bx^\ast$ is a point of attraction of the iteration $\bx^{(k+1)}=\bG_h(\bx^{(k)})$, i.e. there is a neighborhood $D$ of $\bx^\ast$ such that the iterative method converges to the solution $\bx^\ast$ for any initial guess in $D$ with $\lim_{k\to\infty}\|\bx^{(k)}-\bx^\ast\|=0$. 

By the analyticity of $F_i$ we have that there is a constant $C_0>0$ independent of $h$ such that $\|\bDF(\bx)\|\leq C_0$ for all $\bx\in U$. 
Using (\ref{eq:limit}), the continuity of $\bJ_h(\bx)$ and the fact that $\bJ_j(\bx^\ast)$ is invertible, we deduce that for large $k$ there is an $h_2>0$ such that for $0<h<h_2$ and $C_2,C_3>0$ such that $\|\bJ_h(\bx^{(k)})\|\leq C_2$ and $\|\bJ_h^{-1}(\bx^{(k)})\|\leq C_3$. Using Lemma \ref{lem:ostr1a} and taking $0<h<h_0$ with $h_0=\min(h_1,h_2)$ we have
\begin{align*}
\|\bx^{(k+1)}-\bx^\ast\| &=\|\bx^{(k)}-\bJ_h^{-1}(\bx^{(k)})\bF(\bx^{(k)})-\bx^\ast\|\\
&\leq \|-\bJ_h^{-1}(\bx^{(k)}) \left(\bF(\bx^{(k)})-\bF(\bx^\ast)-\bJ_h(\bx^{\ast})(\bx^{(k)}-\bx^\ast) \right)\|\\
&\quad + \|\bJ_h^{-1}(\bx^{(k)})(\bJ_h(\bx^{(k)})-\bJ_h(\bx^\ast))(\bx^{(k)}-\bx^\ast)\|\\
&\leq \|\bJ_h^{-1}(\bx^{(k)}) \|~ \|\bF(\bx^{(k)})-\bF(\bx^\ast)-\bJ_h(\bx^\ast)(\bx^{(k)}-\bx^\ast)\|\\
&\quad + \|\bJ_h^{-1}(\bx^{(k)})\|\|(\bJ_h(\bx^{(k)})-\bJ_h(\bx^\ast))(\bx^{(k)}-\bx^\ast)\|\\
&\leq C \left( \|\bx^{(k)}-\bx^\ast\|^2+h^2 \|\bx^{(k)}-\bx^\ast\| \right)\ .
\end{align*}
This implies that 
\begin{equation}\label{eq:limit1}
\frac{\|\bx^{(k+1)}-\bx^\ast\|}{\|\bx^{(k)}-\bx^\ast\|}\leq C\left(\|\bx^{(k)}-\bx^\ast\|+h^2 \right)\ ,
\end{equation}
and
\begin{equation}\label{eq:limit2}
\frac{\|\bx^{(k+1)}-\bx^\ast\|}{\|\bx^{(k)}-\bx^\ast\|^2}\leq C\left(1+h^2\frac{1}{\|\bx^{(k)}-\bx^\ast\|} \right)\ .
\end{equation}
Taking the limit $k\to\infty$ in (\ref{eq:limit1}) we obtain linear convergence for small values of $h$. Moreover, taking first the limit $h\to 0$ and then the
limit $k\to\infty$ we obtain the quadratic convergence.
\end{proof}

In the case of the complex-step Jacobian-free Newton method, the situation is different. 
Similarly to Lemma \ref{lem:ostr1a}, we have the following:

\begin{lemma}\label{lem:ostr1}
If $\bF:U\subset \mathbb{R}^n\to \mathbb{R}^n$ has real-analytic entries in a neighborhood $U\subset \mathbb{R}^n$, then there is a neighborhood $D\subset \mathbb{R}^n$ and $h_0>0$ such that for any $\bx$, $\by\in D$ and $0<h<h_0$ 
\begin{equation}\label{eq:ostr2}
\|\bF(\by)-\bF(\bx)-\bD_h\bF(\bx)(\by-\bx)\|\leq C\left(\|\by-\bx\|^2+h^2\|\by-\bx\|^3\right)\ ,
\end{equation}
for some positive constant $C$ independent of $h$.
Moreover, there is a constant $C'$ such that 
\begin{equation}\label{eq:ostr3}
\|\bD_h\bF(\bx)(\by-\bx)-\bD_h\bF(\by)(\by-\bx)\|\leq C'\left(\|\by-\bx\|^2+h^2\|\by-\bx\|^3\right)\ .
\end{equation}
\end{lemma}
\begin{proof}
With the help of (\ref{eq:cstepvec}) we write the quantity in (\ref{eq:ostr2}) as
$$\bF(\by)-\bF(\bx)-\bD_h\bF(\bx)(\by-\bx)=\bF(\by)-\bF(\bx)-\bDF(\bx)(\by-\bx)+h^2\bR^\bx(c_h;\by-\bx)\ ,$$
where $\|\bR^\bx(c_h;\by-\bx)\|\leq C\|\by-\bx\|^3$ because of (\ref{eq:normres}).
Note that (cf. \cite[Lemma 8.1.2]{ortega1990})
$$\bF(\by)-\bF(\bx)-\bDF(\bx)(\by-\bx)=\int_0^1[\bDF(\bx+t(\by-\bx))-\bDF(\bx)](\by-\bx)~dt\ .$$
Therefore, we have
$$\begin{aligned}
\|\bF(\by)-\bF(\bx)-\bD_h\bF(\bx)(\by-\bx)\|&\leq \| \bF(\by)-\bF(\bx)-\bDF(\bx)(\by-\bx)\|+ C h^2\|\by-\bx\|^3\\
&\leq \int_0^1\|\bDF(\bx+t(\by-\bx))-\bDF(\bx)\|\|\by-\bx\|dt+ C h^2\|\by-\bx\|^3\\
&\leq C\left(\|\by-\bx\|^2+h^2\|\by-\bx\|^3\right)\ .
\end{aligned}$$
Similarly, using again (\ref{eq:cstepvec}) we have
$$\begin{aligned}
\|\bD_h\bF(\bx)(\by-\bx)-\bD_h\bF(\by)(\by-\bx)\| \leq & \|(\bDF(\bx)-\bDF(\by))(\by-\bx)\|\\
&+h^2\|\bR^\bx(c_h;\by-\bx)-\bR^\by(c_h;\by-\bx)\|\ ,
\end{aligned}
$$
from which we obtain the inequality (\ref{eq:ostr3}).
\end{proof}

Specifically, we have the following theorem:

\begin{theorem}\label{thm:jacfreeconv}
Let the function $\bF=(F_1,F_2,\dots,F_n):U\subset \mathbb{R}^n\to \mathbb{R}^n$ with $F_i$ real analytic functions in a closed domain $U$ and $\bx^\ast$ is the unique vector in the interior of $U$ such that $\bF(\bx^\ast)=\bo$. If the matrix $\bDF(\bx^\ast)$ is invertible, then there is a neighborhood $D\subset \mathbb{R}^n$ of $\bx^\ast$ and $h_0>0$ such that for any initial guess $\bx^{(0)}\in D$ and for all $0<h<h_0$, the complex-step Jacobian-free Newton iteration (\ref{eq:twodcsnewton}) is well-defined and converges to $x^\ast$. Moreover, there is a constant $C\geq0$ independent of $h$ such that
\begin{equation}\label{eq:errest}
\lim_{k\to \infty}\frac{\|\bx^{(k+1)}-\bx^\ast\|}{\|\bx^{(k)}-\bx^\ast\|^2}=C\ ,
\end{equation}
for all $0<h<h_0$.
\end{theorem}

\begin{proof}
We define 
$$\mathcal{F}(\bx,\by)=\frac{1}{h}\Im \bF(\bx+\I h\by)-\bF(\bx)\ .$$
Then, we write the Jacobian-free Newton method as
$$\bx^{(k+1)}=\bx^{(k)}-\bu^{(k)},\quad \text{for $k=0,1,\dots$}\ ,$$
where $\bu^{(k)}$ is the solution of the equation $\mathcal{F}(\bx^{(k)},\bu^{(k)})=0$.

Due to the analyticity of $\bF$ we have that 
$$\bD_\bx\mathcal{F}(\bx,\by)=\frac{1}{h}\Im \bDF(\bx+\I h\by)-\bDF(\bx)\ ,$$ 
and
$$\bD_\by\mathcal{F}(\bx,\by)=\Re \bDF(\bx+\I h\by) \ ,$$
and specifically $\bD_\bx\mathcal{F}(\bx^\ast,\bo)=-\bDF(\bx^\ast)$ and $\bD_\by\mathcal{F}(\bx^\ast,\bo)=\bDF(\bx^\ast)$, are invertible. Applying the implicit function theorem \cite{B1977}, we deduce that there is a neighborhood $D\subset \mathbb{R}^n$ containing $\bx^\ast$ and a unique function $\bg(\bx)$ such that $\bg(\bx^\ast)=\bo$ and $\mathcal{F}(\bx,\bg(\bx))=0$ for all $\bx\in D$. Moreover, $\bD \bg(\bx)=-[\bD_\by \mathcal{F}(\bx,\bg(\bx))]^{-1} \bD_\bx \mathcal{F}(\bx,\bg(\bx))$, and $\bD\bg(\bx^\ast)=\bI$.

With the help of the function $\bg$, we express the complex-step Jacobian-free Newton method in the form of the fixed point iteration:
\begin{equation}\label{eq:JFNM}
\bx^{(k+1)}=\bG_h(\bx^{(k)})\doteq \bx^{(k)}-\bg(\bx^{(k)}),\quad \text{for $k=0,1,\dots$}\ ,
\end{equation}
Due to the analytic properties of $\bF$, we deduce that the complex-step Jacobian-free Newton method is well-defined at least in a neighborhood $D\subset \mathbb{R}^n$ that contains the root $\bx^\ast$. Also, we have that the Fr\'{e}chet derivative of $\bG_h$ at $\bx^\ast$ is the zero matrix
$$\bG'_h(\bx^\ast)=\bI-\bD\bg(\bx^\ast)=\bo\ .$$

Therefore, from Ostrowski's Theorem \cite[Theorem 8.1.7]{ortega1990} we have that $\bx^\ast$ is a point of attraction of the iteration $\bx^{(k+1)}=\bG_h(\bx^{(k)})$, i.e. there is a neighborhood $D$ of $\bx^\ast$ such that the iterative method converges to the solution $\bx^\ast$ for any initial guess in $D$ with $\lim_{k\to\infty}\|\bx^{(k)}-\bx^\ast\|=0$. 

Since $\bDF(\bx^\ast)$ is not singular, we have that $\bDF(\bx)$ in a neighborhood of $\bx^\ast$ is also not singular \cite{ortega1990}. Therefore, from (\ref{eq:approx1}) there is an $h_0>0$ such that for all $0<h<h_0$ and for large $k$ and $\by$ close to $\bo$, the derivative of $\bD_\by [\bD_h\bF(\bx^{(k)})\by]=\Re \bDF(\bx^{(k)}+\I h\by)$ is not singular. By the Inverse Function Theorem \cite{B1977}, we have that $[\bD_h\bF(\bx^{(k)})]^{-1}$ also exists and is analytic in a neighborhood around $\bo$. Therefore, we can write (\ref{eq:JFNM}) in the form \begin{equation}\label{eq:JFNM2}
\bx^{(k+1)}=\bx^{(k)}-[\bD_h\bF(\bx^{(k)})]^{-1} \bF(\bx^{(k)})\ .
\end{equation}
Moreover, since $[\bD_h\bF(\bx^{(k)})]^{-1}$ is analytic in a neighborhood $U$ containing $\bo$, it is also Lipschitz, and for any $\bu,\bv\in U$, there is a constant $C>0$ such that
\begin{equation}\label{eq:lips}
\|[\bD_h\bF(\bx^{(k)})]^{-1}\bu-[\bD_h\bF(\bx^{(k)})]^{-1}\bv\|\leq C \|\bu-\bv\|\ .  
\end{equation}

Using (\ref{eq:JFNM2}), (\ref{eq:lips}), Lemma \ref{lem:ostr1} and taking $0<h<h_0$, we have
\begin{align*}
\|\bx^{(k+1)}-\bx^\ast\| &=\|\bx^{(k)}-[\bD_h\bF(\bx^{(k)})]^{-1}\bF(\bx^{(k)})-\bx^\ast\|\\
&= \|[\bD_h\bF(\bx^{(k)})]^{-1}\left(\bD_h\bF(\bx^{(k)})(\bx^{(k)}-\bx^\ast) \right)-[\bD_h\bF(\bx^{(k)})]^{-1} \left(\bF(\bx^{(k)})-\bF(\bx^\ast)\right)\|\\ 
&\leq C~ \|\bF(\bx^{(k)})-\bF(\bx^\ast)-\bD_h\bF(\bx^{(k)})(\bx^{(k)}-\bx^\ast)\|\\
&\leq C~ \left(\|\bF(\bx^{(k)})-\bF(\bx^\ast)-\bD_h\bF(\bx^\ast)(\bx^{(k)}-\bx^\ast)\|\right.\\
&\quad + \left.\|\bD_h\bF(\bx^{(k)})(\bx^{(k)}-\bx^\ast)-\bD_h\bF(\bx^\ast)(\bx^{(k)}-\bx^\ast)\|\right)\\
&\leq C \left( \|\bx^{(k)}-\bx^\ast\|^2+h^2 \|\bx^{(k)}-\bx^\ast\|^3 \right)\ .
\end{align*}
This implies that for large $k$
$$
\frac{\|\bx^{(k+1)}-\bx^\ast\|}{\|\bx^{(k)}-\bx^\ast\|^2}\leq C\left(1+h^2\|\bx^{(k)}-\bx^\ast\| \right)\ .
$$
Taking the limit $k\to\infty$, we obtain that (\ref{eq:errest}) holds, and the rate is quadratic for all $0<h<h_0$.
\end{proof}

\subsection{Experimental convergence rates}

The implementation of the complex-step Jacobian Newton method is straightforward and involves the assembly of the approximate Jacobian matrix at every iteration and the solution of the corresponding linear systems with any method of our choice. The implementation of the complex-step Jacobian-free Newton method can follow the Newton-Krylov methodology \cite{BS1990,BH1986}, where the nonlinear term $\bD_h\bF(\bx^{(k)})\bu^{(k)}$ is considered as an approximation of the matrix-vector product $\bDF(\bx^{(k)})\bu^{(k)}\approx \frac{1}{h} \Im \bF(\bx+\I h\bv)$ in Newton's method, and the system $\bDF(\bx^{(k)})\bu^{(k)}=\bF(\bx^{(k)})$ can be solved approximately by Krylov methods without the knowledge the Jacobian matrix $\bDF$ but only of the corresponding matrix-vector product. As a result, we obtain an iterative scheme that converges with quadratic rate.

To explore the differences between the two variants of the complex-step Newton method we consider the following system of equations:
\begin{equation}\label{eq:systemv}
\begin{aligned}
&x_1(e^{x_1/2}+1)=0\ ,\\
&x_2(e^{x_2/2}+1)=0\ ,
\end{aligned}
\end{equation}
which has exact solution $(x_1,x_2)=(0,0)$. The two equations are uncoupled and the complex-step Jacobian Newton method is reduced to solving two scalar equations, while the Jacobian-free variant limited by the nonlinear coupling of the operator $\bD_h\bF$ differs from the scalar case.

\begin{figure}[ht!]
  \centering
\includegraphics[width=0.8\columnwidth]{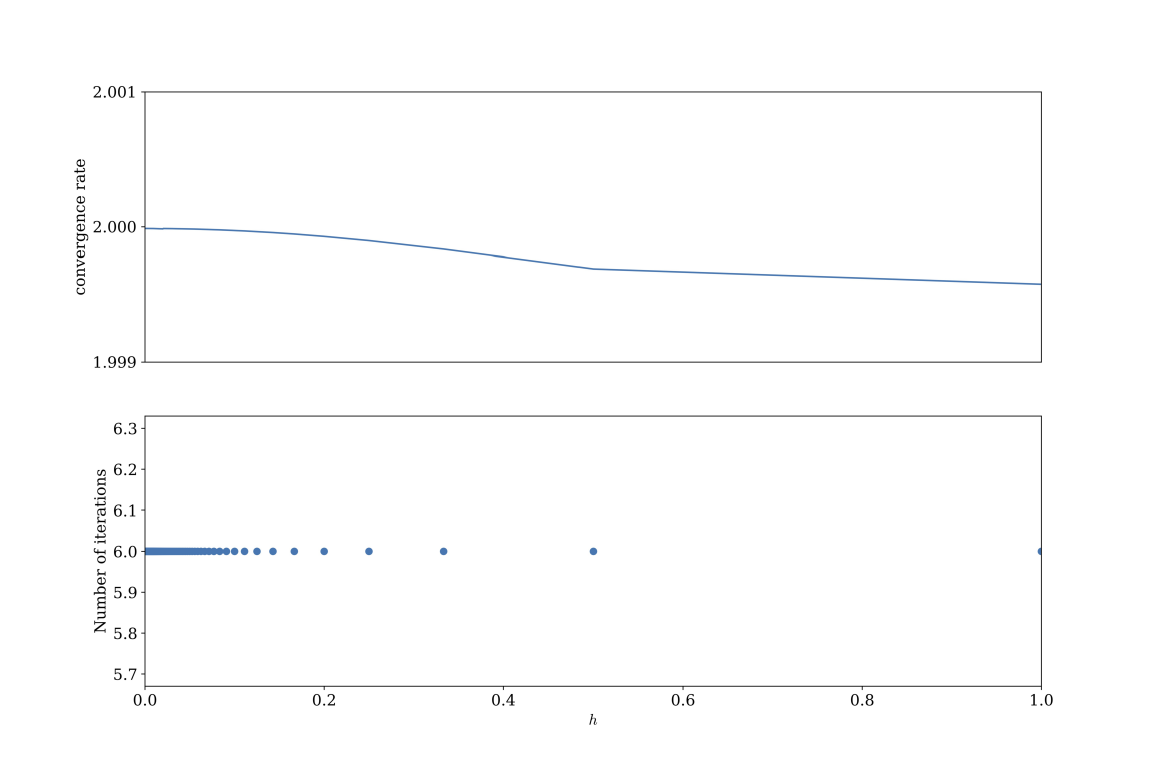}
  \caption{Convergence rate and number of iterations as a function of $h$ of the complex-step Jacobian-free Newton method}
  \label{fig:rates2}
\end{figure}

In this experiment we used initial guess $(x_1^{(0)},x_2^{(0)})=(2.5,2.5)$. For the approximation of the solution of system (\ref{eq:linearsysc}), we employed the accelerated restarted GMRES method \cite{BJM2005} as implemented in the function \texttt{lgmres} of Python  with convergence tolerance $10^{-14}$. The details of the convergence of the complex-step Jacobian-free Newton method are presented in Figure \ref{fig:rates2}. It is noted that the number of iterations required by the GMRES method to converge within the prescribed tolerance varied with the values of $h$ and $k$. Specifically, the number of iterations decreased as $h$ decreased and also as $k$ increased, with the requirement of only a couple of iterations after the first few values of $h$. Figure \ref{fig:rates2} presents the experimental convergence rate $\log[e_{k+1}/e_{k}]/\log[e_{k}/e_{k-1}]$ where $e_k=\|\bx^{(k)}-\bx^\ast\|$ using the last three iterations before the error $e_k$ becomes less than the tolerance $10^{-14}$. The choice of the tolerance was taken so as to achieve clear convergence rates for large values of $k$. In the same figure we present the required number of iterations as function of $h$. We observe that the complex-step Jacobian-free Newton method converges at quadratic rate and within 6 iterations for all values of $h$ we tested in the interval $[10^{-3},1]$. The complex-step Jacobian Newton method has exactly the same convergence properties as in the scalar equation, and the convergence rate along with required iterations are the same as in Figure \ref{fig:rates1}.

\section{Performance in applications}\label{sec:application}

By demonstrating that, in practice, the complex-step Jacobian-free Newton method converges with a quadratic rate, and that the iteration does not necessitate knowledge of any derivatives or Jacobian matrices, we can conclude that it serves as a superior alternative to the classical Secant method. Furthermore, the fact that the evaluation of the complex-step approximation of the Jacobian does not require knowledge of the corresponding matrix places this method within the broader category of Jacobian-free Newton-Krylov methods \cite{KK2004,Kelley1995} and it has been demonstrated by various problems in \cite{KSPC22}. Specifically, considering the approximation of the Jacobian matrix 
\begin{equation}\label{eq:jac2}
\bDF(\bx)\by\approx \frac{1}{h}\Im \bF(\bx+\I h \by)\ ,
\end{equation}
we can compute an approximate solution of the corresponding linear system using Krylov subspaces methods \cite{Saad2003}. In the following experiments we used the accelerated restarted GMRES method of \cite{BJM2005} as it is implemented in the \texttt{lgmres} function of Python for the solution of the linear systems and without assembling the Jacobian matrix but by using the formula (\ref{eq:jac2}). 

\subsection*{Solution of ordinary differential equations}

We start by reporting on the convergence of the complex-step Jacobian-free Newton method applied to the solution of the nonlinear algebraic systems obtained by approximating numerically the solutions of two initial-value problems $y'(t)=f(y(t),t)$. For the time-integration of the ordinary differential equations we consider the fourth-order Gauss-Legendre Runge-Kutta method with two stages given by the Butcher {\em tableau}
\begin{equation}\label{eq:Butcher}
\begin{array}{c|c}
\bc & \bA \\
\hline 
\rule{0pt}{3ex}  & \bb^T 
\end{array}\quad =\quad 
\begin{array}{c|cc}
\frac{1}{2}-\frac{1}{6}\sqrt{3} & \frac{1}{4} & \frac{1}{4}-\frac{1}{6}\sqrt{3} \\
\rule{0pt}{3ex} \frac{1}{2}+\frac{1}{6}\sqrt{3} & \frac{1}{4}+\frac{1}{6}\sqrt{3} & \frac{1}{4} \\ [3pt]
 \hline
\rule{0pt}{3ex}  & \frac{1}{2} & \frac{1}{2}
\end{array}
\end{equation}
Using stepsize $\Delta t$ this particular method requires at every step the solution of the nonlinear system:
\begin{equation}\label{eq:nonlins}
\begin{aligned}
&k_1=f(t_i+c_1h,~y_i+\Delta t (a_{11}k_1+a_{12}k_2))\ ,\\
&k_2=f(t_i+c_2h,~y_i+\Delta t (a_{21}k_1+a_{22}k_2))\ ,
\end{aligned}
\end{equation} 
for the computation of the intermediate stages $k_1$ and $k_2$. After solving the nonlinear system
\begin{equation}\label{eq:nonlins2}
\bF(k_1,k_2)\doteq
\begin{pmatrix}
f(t_i+c_1h,~y_i+\Delta t (a_{11}k_1+a_{12}k_2))-k_1\\
f(t_i+c_2h,~y_i+\Delta t (a_{21}k_1+a_{22}k_2))-k_2
\end{pmatrix}
=\begin{pmatrix}
    0 \\ 0
\end{pmatrix}\ ,
\end{equation} 
for the intermediate stages $k_1$ and $k_2$, we obtain the approximate solution of $y(t^{i+1})$ for the time $t^{i+1}=t^i+\Delta t$ by the formula
$y_{i+1}=y_i+\Delta t (b_1k_1+b_2k_2)$.  For more information related to the particular stiff differential equation and the Runge-Kutta method we refer to the book \cite{HNW1993ii}.

\begin{figure}[ht!]
  \centering
\includegraphics[width=0.8\columnwidth]{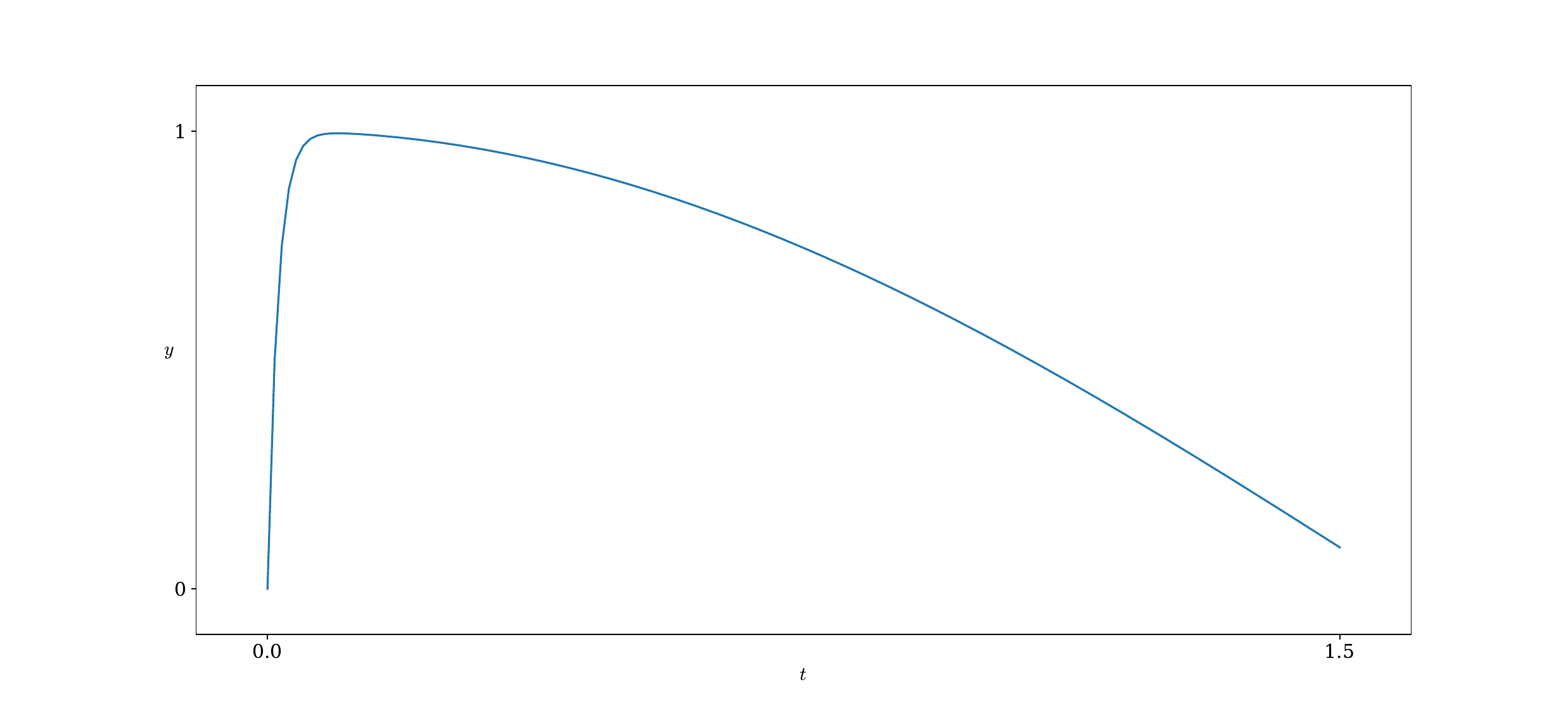}
  \caption{Numerical solution of a stiff ordinary differential equation obtained by solving the nonlinear systems using the complex-step Jacobian-free Newton method}
  \label{fig:ode}
\end{figure}

First we solve the initial value problem $y'(t)=f(y(t),t)$ with $y(0)=0$ and $f(y,t)=-50(y-\cos ~t)$. This is a stiff ordinary differential equation where its numerical approximation requires appropriate numerical methods.

In this experiment we took $\Delta t=10^{-2}$ and for $h=1/n$ for $n=1,2,\dots,10^{6}$ we obtained the solution as depicted in Figure \ref{fig:ode} with extremely fast convergence that required $2$ iterations for all values of $h$ and for all stages of the Runge-Kutta method indicating high-order convergence rate. At every iteration we used the solution of the previous step as an initial condition for the next complex-step Newton iteration. The tolerance we used for this case was $10^{-12}$ for determining the convergence of the Newton method, and $10^{-12}$ for the convergence of the GMRES method. In this experiment, the accelerated restarted GMRES method required at most $2$ iterations to converge for all values of $h$ and for all timesteps.

\begin{figure}[ht!]
  \centering
\includegraphics[width=0.5\columnwidth]{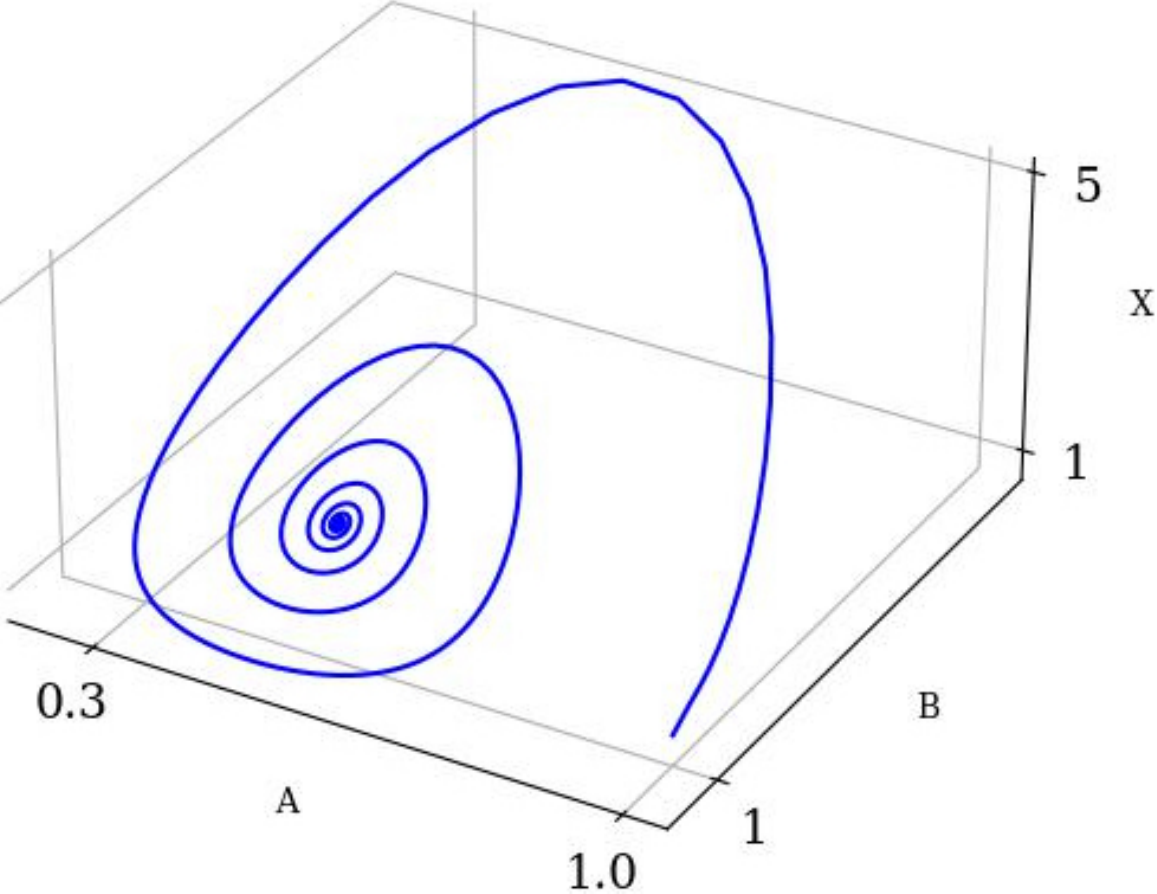}
  \caption{Projection of the spiral orbit of the Olsen system in the space $A-B-X$}
  \label{fig:Olsen}
\end{figure}

We also considered a 4D dynamical system known as Olsen model for biochemical peroxidase-oxidase reaction \cite{Olsen1983}. A characteristic property of this system is that one component of its solution varies much slower than the other three. The Olsen system
\begin{equation}\label{eq:Olsen}
\begin{aligned}
    \frac{dA}{dt} &= \mu-\alpha A-ABY\ ,\\
    \frac{dB}{dt} &= \varepsilon(1-BX-ABY)\ ,\\
    \frac{dX}{dt} &= \lambda(BX-X^2+3ABY-\zeta X+\delta)\ ,\\
    \frac{dY}{dt} &= \kappa\lambda(X^2-Y-ABY)\ , 
\end{aligned}
\end{equation}
where we used the parameters  $\alpha=0.0912$, $\beta=1.2121\times 10^{-5}$, $\varepsilon=0.0037$, $\lambda=18.5281$, $\kappa=3.7963$, $\mu=0.9697$, $\zeta=0.9847$. This in combination with the initial conditions $(A(0),B(0),X(0),Y(0))=(1,1,1,1)$ leads to a spiral solution in the four-dimensional phase-space \cite{MKH2020}. Taking $t\in [0,10]$ and $\Delta t=0.01$ we obtained the solution where its projection in the space $A-B-X$ is depicted in Figure \ref{fig:Olsen}. The particular solution was obtained with $h=0.1$ and tolerance for the complex-step Jacobian-free Newton method $10^{-12}$. As it was expected, the complex-step Jacobian-free Newton method converged within 4 iterations for all values of $h<1$ that we used. This verifies the theory presented in the previous section. In this experiment, we used the accelerated restarted GMRES method and tolerance $10^{-12}$ which required at most 3  iterations for convergence.

\subsection*{Solution of the DNLS}

Here we consider a more complicated example. Specifically, we demonstrate the ability of the complex-step Jacobian-free Newton method to solve complex equations in a case of significant interest by solving the corresponding real and imaginary parts of the equations as real equations. In particular, we consider the computation of ground states of the Discrete Nonlinear Schr\"{o}dinger (DNLS) equation which is a system of complex nonlinear ordinary differential equations with applications in biology, optics, plasma physics and other fields of science. For more information about the particular equation we refer the interested reader to the book \cite{Kev}. 

We consider the DNLS equation 
\begin{equation}
    \label{eq:dnls}
    \I \dot{u}_n+(u_{n+1}-2 u_n+ u_{n-1})+|u_n|^2 u_n =0\ ,
\end{equation}
where $u_n=u_n(t):\mathbb{R}\to\mathbb{C}$ can be thought of as an approximate solution of the Nonlinear Schr\"{o}dinger equation at the lattice sites $n=1,2,\dots, N+1$ with periodic boundary condition $u_1(t)=u_{N+1}(t)$. In the previous notation $\I=\sqrt{-1}$ and $\dot{u}_n=\tfrac{d}{dt}u_n(t)$. Solutions of the DNLS equation conserve the Hamiltonian 
\begin{equation}
    H(t)=-\sum_{n=1}^N [|u_n-u_{n-1}|^2-\frac{1}{2}|u_n|^4]\ ,
\end{equation}
and the norm
\begin{equation}
    P(t)=\sum_{n=1}^N|u_n|^2\ .
\end{equation}

First we search for {\em steady state} solutions of the form $u_n=e^{\I\omega t}v_n$, where $\omega\in\mathbb{R}$ and $v_n\in\mathbb{C}$ independent of $t$. Substitution of this {\em ansatz} into the DNLS equation (\ref{eq:dnls}) yields a system of complex equations 
\begin{equation}
    \label{eq:steadeq}
    -\omega v_n+(v_{n+1}-2 v_n+ v_{n-1})+|v_n|^2 v_n = 0,\quad n=1,2,\dots, N\ .
\end{equation}
Denoting $v_n=x_n+\I y_n$, we write the system (\ref{eq:steadeq}) in the form $\bF(\bx)=\bo$ where $$\bx=(x_1,\dots,x_N,y_1,\dots, y_N)^T$$ and $\bF:\mathbb{R}^{2N}\to \mathbb{R}^{2N}$ with $\bF(\bx)=(\bX(\bx),\bY(\bx))^T$ where $\bX,\bY:\mathbb{R}^N\to \mathbb{R}^N$ are real vectors with entries 
$$ X_j = -\omega x_j+(x_{j+1}-2 x_j+ x_{j-1})+(x_j^2+y_j^2) x_j\ , $$
and 
$$Y_j = -\omega y_j+(y_{j+1}-2 y_j+ y_{j-1})+(x_j^2+y_j^2) y_j\ ,$$ for $j=1,2,\dots, N$. Note that in the previous expressions periodic boundary conditions must be employed for $j=1$ and $j=N$ so as the equations to make sense.

Setting the tolerance for convergence to be $\|\bx^{(k+1)}-\bx^{(k)}\|<10^{-12}$, $h$ as before, and taking a lattice with $N=200$ points, we obtained the numerical ground state solution for $\omega=0.1$ within 8 iterations independent of the choice of $h\leq 0.1$. Specifically, we considered values of $h=1/k$ for $k=10,11,\dots,10^3$. The accelerated restarted GMRES method required initially 9 iterations, which decreased gradually and stabilized at 4 iterations after the value $h=1/55$. As initial guess of the solution we used the soliton-like profile
$$v_n^{(0)}=\frac{1+\I}{2}{\sech}^2(n-n_0)\ ,$$
where $n_0=100$. The computed solution is presented in Figure \ref{fig:solution}(a).

\begin{figure}[ht!]
  \centering
\includegraphics[width=\columnwidth]{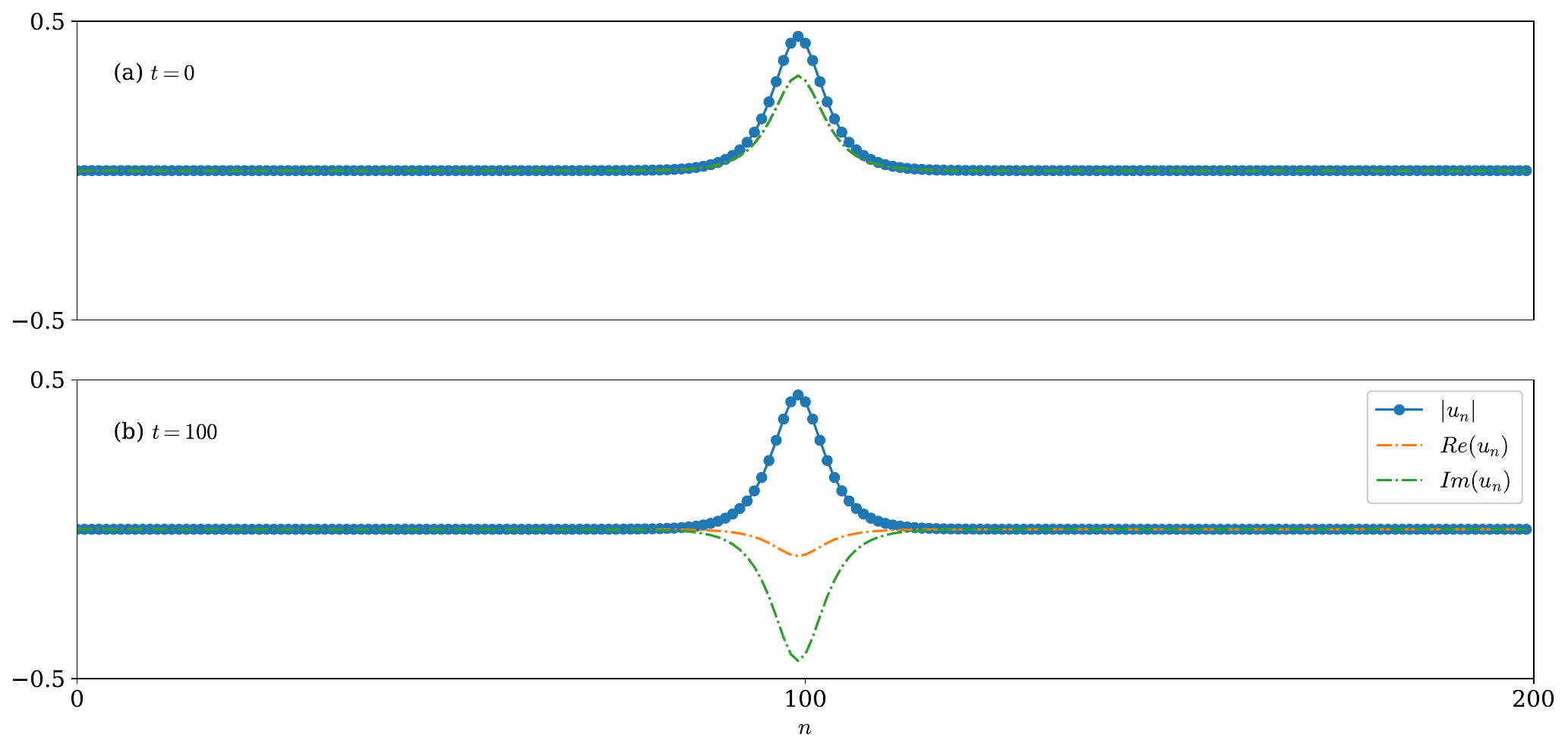}
  \caption{The numerical solution at $t=0$ and $t=100$}
  \label{fig:solution}
\end{figure}
\begin{figure}[ht!]
  \centering
\includegraphics[width=\columnwidth]{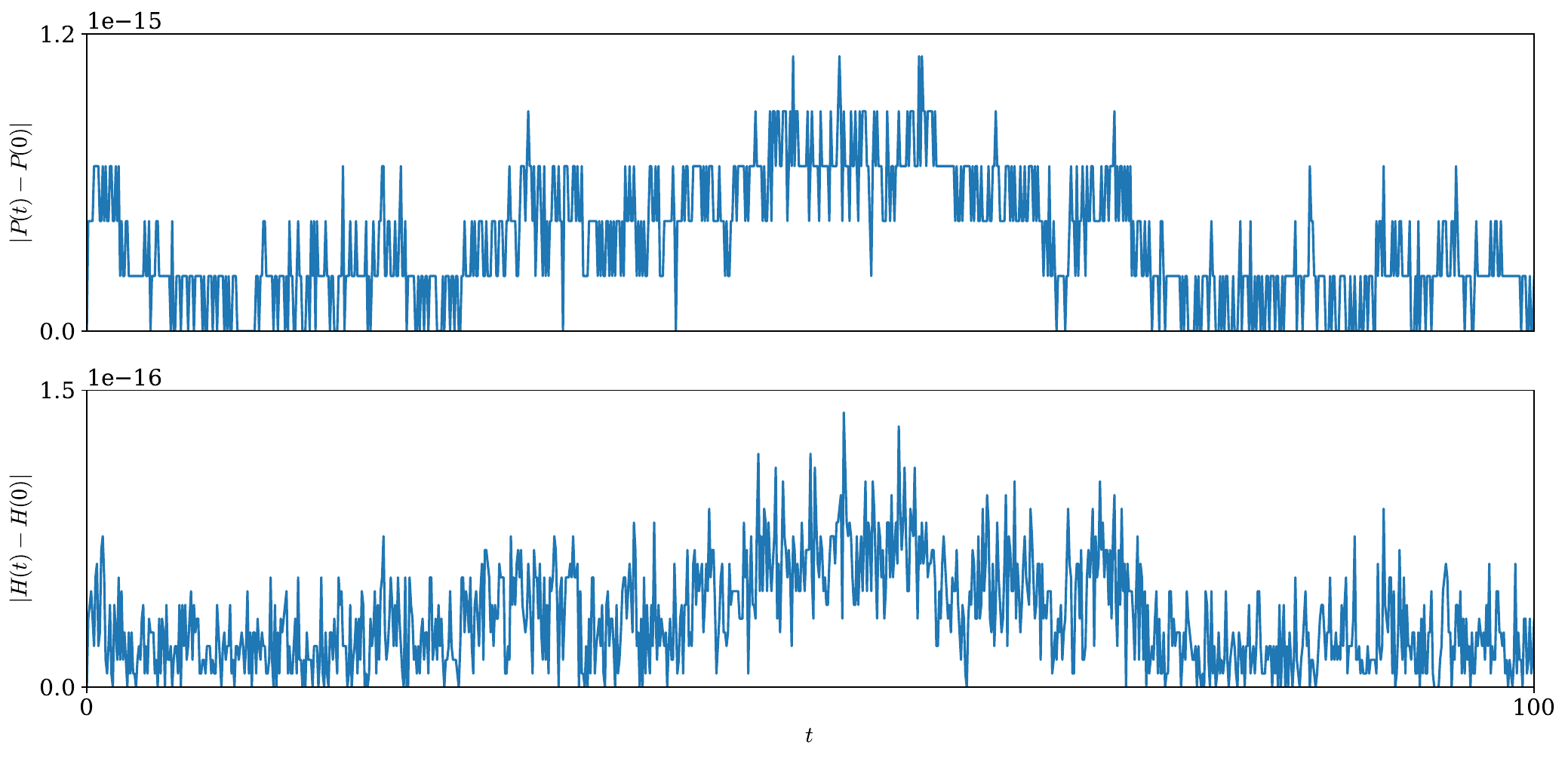}
  \caption{The error in the computed norm and Hamiltonian}
  \label{fig:errors}
\end{figure}

To further assess the convergence of the complex-step Jacobian-free Newton iteration, we approximated numerically the solution of the DNLS equation with initial condition the numerically obtained ground state. For this reason we used again the fourth-order Gauss-Legendre Runge-Kutta method given by the {\em tableau} (\ref{eq:Butcher}). This method has been studied for the numerical solution of the nonlinear Schr\"{o}dinger equation in \cite{ADK}. To apply the complex-step Jacobian-free Newton method, we rewrite the complex system of differential equations into a system of real differential equations. In particular, if we denote $R_n=\Re(u_n)$ and $I_n=\Im(u_n)$ we write system (\ref{eq:dnls}) in the form
\begin{equation}\label{eq:nonlins3}
\begin{aligned}
    \frac{d}{dt}R_n&=-[I_{n+1}-2I_n+I_{n-1}+(R_n^2+I_n^2) I_n]\ ,\\
    \frac{d}{dt}I_n&= R_{n+1}-2R_n+R_{n-1}+(R_n^2+I_n^2) R_n\ ,
\end{aligned}
\end{equation}
and then we apply the complex-step Jacobian-free Newton method to the real nonlinear system (\ref{eq:nonlins3}). For the discretization of (\ref{eq:nonlins3}), we used a final time of $T=100$ and a time step of $\Delta t=0.1$. We tested the behaviour of the complex-step Jacobian-free Newton method for various values of the complex-step $h\leq 1$, and we found that the method converged for all values we tried. Specifically, the complex-step Jacobian-free Newton method converged within 4 iterations for all intermediate stages and for all values of $h$ we used, while we used the tolerance $10^{-12}$ for the complex-step Newton method and $10^{-6}$ for the accelerated restarted GMRES method. The GMRES method converged required maximum 4 iterations for each Newton iteration and the complex-step Jacobian-free Newton method required only $3$ iterations for all time steps. For some values $h>0.1$ that we used, we observed that the method didn't converge with the Python implementation of GMRES method and the tolerances we set but it converged with very similar results with the classical GMRES method of Matlab$^\text{\textregistered}$ for all values $h\leq 1$. The solution at $t=100$ is presented in Figure \ref{fig:solution}(b). The norm was preserved at the value $P=1.25217740220729$, while the Hamiltonian remained at $H=0.041394478367519$. These values indicate an error in norm of the order $10^{-15}$ and an error of $10^{-16}$ in the Hamiltonian. The error in their computation is depicted in Figure \ref{fig:errors}. When we attempted initial conditions that did not correspond to traveling or standing waves, the solution preserved the norm within machine precision, but the Hamiltonian was conserved to an order of $10^{-11}$. This aligns with the nearly preserved Hamiltonian of symplectic methods \cite{HLW2006}. This is an indication that although the numerical method required approximation of solutions through iterative methods (GMRES and complex-step Newton methods), the results remained very accurate. For other applications and examples we refer to \cite{KSPC22}.

\section{Conclusions}

In this work, we considered the complex-step Jacobian-free Newton method. After reviewing the derivation of the complex-step Newton iteration for both scalar and system equations, we proved its convergence for both the Jacobian and Jacobian-free variants. Specifically, we established that the convergence in the Jacobian case is linear, but it becomes quadratic as the complex-step parameter $h$ tends to zero. In contrast, the Jacobian-free variant converges quadratically for any appropriately small value of $h>0$. We concluded this work with experimental studies on the convergence of the complex-step iteration. Additionally, we tested the method on the numerical solution of stiff ordinary differential equations using a symplectic Runge-Kutta method, and we examined the discretization of the Nonlinear Schr\"{o}dinger equation with the same symplectic method. In both cases, the complex-step Jacobian-free Newton iteration performed seamlessly.


\end{document}